\documentclass[11pt,twoside, reqno]{article}

\usepackage{amssymb}
\usepackage{amsmath}
\usepackage{mathrsfs}
\usepackage{amsthm}
\usepackage{txfonts}
\usepackage{graphicx}
\usepackage{graphics}
\usepackage{subfigure}

\newtheorem{theorem}{Theorem}[section]
\newtheorem{lemma}[theorem]{Lemma}
\newtheorem{corollary}[theorem]{Corollary}

\theoremstyle{definition}
\newtheorem{remark}[theorem]{Remark}
\newtheorem{definition}[theorem]{Definition}
\newtheorem{example}[theorem]{Example}

\begin{document}

\title{\Large\bf Functional inequalities for starlike and convex functions associated with $g$-derivative operator \footnotetext{\hspace{-0.35cm} 2020 {\it
Mathematics Subject Classification}. {26A33; 30C45; 30C50; 33D05.}
\endgraf{\it Key words and phrases}. Starlike function, convex function, $g$-derivative operator, Fekete-Szeg\"{o} inequality, Toeplitz determinant, Hankel determinant.
\endgraf This work is supported by Natural Science Foundation of China(Grant Nos.\
12261068).}}
\author{An Huang; Pinhong Long; Halit Orhan; Huo Tang}
\date{ }
\maketitle


\begin{center}
\begin{minipage}{13.5cm}\small
{{\bf Abstract.} In this paper we first introduce a new class of derivative operators, termed $g$-derivative operators, which unifies the $q$-derivative, $(p,q)$-derivative and $(\alpha,\beta,\gamma)$-derivative operators, 
even classical derivative in the literature. Based on this generalized framework, we define two novel function classes, specifically $g$-starlike functions and $g$-convex functions. Further, 
we employ the subordination principle of analytic functions to conduct the coefficient estimations for such function classes, and subsequently derive the bounds for the corresponding Fekete-Szeg\"{o} inequality, 
Toeplitz determinants and Hankel determinants.}
\end{minipage}
\end{center}

\section{Introduction\label{s1}}

Let $\mathbb{A}$ be the subset of $H(\mathbb{U})$ which includes all functions satisfying the standardized conditions
\begin{align}
f(0)={f}' (0)-1=0\nonumber,
\end{align}
here $H(\mathbb{U})$ denotes the set of all analytic functions on the open unit disk
\begin{align}
\mathbb{U} = \left \{ z\ \in \mathbb{C}: \ \ \left | z \right | \ < 1 \right \}\nonumber.
\end{align}
Further, for any $f\in \mathbb{A} $, its Taylor-Maclaurin expansion is given with
\begin{align}\label{e1.1}
f(z)=z+\sum^{\infty}_{n=2}a_{n}z^{n}.
\end{align}
Additionally, the subclass $\mathbb{S}$ of $\mathbb{A}$ is defined as the set of all univalent analytic functions in $\mathbb{U}$.

For two analytic functions $f$ and $g$ in $\mathbb{U}$, if they satisfy the following relation of partial order: 
\begin{align}
g(z)\prec f(z)\ \  (g\prec f)\nonumber,
\end{align}
then we call that $g$ is subordinate to $f$ $^{\cite{1Miller}}$.

Furthermore, if $g$ is a univalent analytic functions in $\mathbb{U}$, there has an equivalent relation $^{\cite{2Miller}}$
\begin{align}
g(z)\prec f(z),\ \ (z\in \mathbb{U} )\Leftrightarrow g(\mathbb{U})\subset f(\mathbb{U}),\ f(0)=g(0) \nonumber.
\end{align}

For $f\in \mathbb{A} $, if
\begin{align}
\mathfrak{R} \left ( \frac{z{f}'\left ( z \right )  }{f\left ( z \right ) }  \right ) >0,\ \ \ \left ( z\in \mathbb{U}  \right )  \nonumber
\end{align}
holds, then we call that $f$ belongs to the class $\mathcal{S}^{\ast } $ of starlike functions.

Similarly, for $f\in \mathbb{A} $, if 
\begin{align}
\mathfrak{R} \left (1+ \frac{z{f}'' \left ( z \right )  }{{f}' \left ( z \right ) }  \right ) >0,\ \ \ \left ( z\in \mathbb{U}  \right )  \nonumber
\end{align}
then we call that $f$ belongs to the class $\mathcal{C} $ of convex functions.

Through certain subordination principles of analytic functions, Ma-Minda$^{\cite{6George}}$ defined the class $\mathcal{S}^{\ast } \left ( \phi  \right ) $ of starlike functions and the class $\mathcal{C} \left ( \phi  \right ) $ of convex functions by
\begin{align}
\mathcal{S}^{\ast }  \left ( \phi  \right ) =\left \{ f\in \mathbb{A}:\frac{z{f}'(z) }{f(z)}  \prec \phi (z) \right \} \nonumber
\end{align}
and
\begin{align}
\mathcal{C}\left ( \phi  \right ) =\left \{ f\in \mathbb{A}:1+\frac{z{f}'' (z) }{{f}' (z)}  \prec \phi (z) \right \}, \nonumber
\end{align}
respectively, where $\phi (z)$ is analytic in $\mathbb{U}$, and satisfies $\mathfrak{R} \left ( \phi (z) \right ) >0, \phi (0)=1$.

For $f\in \mathbb{A} $, Thomas and Halim$^{\cite{8Thomas}}$ ever introduced the following symmetric Toeplitz determinant
$$\begin{gathered}\mathcal{T}_j(n)=\begin{vmatrix}a_n&a_{n+1}&\cdots&a_{n+j-1}\\a_{n+1}&a_n&\cdots&a_{n+j}
\\\vdots&\vdots&\ddots&\vdots\\a_{n+j-1}&a_{n+j}&\cdots&a_n\end{vmatrix}\end{gathered}$$
for $j,n\in\mathbb{N}$. Note that 
$$\mathcal{T}_3(1)=\begin{vmatrix}1&a_2&a_3\\a_2&1&a_2\\a_3&a_2&1\end{vmatrix}=(1+a_3-2a_2^2)(1-a_3).$$

Besides, for $f\in \mathbb{A} $, Pommerenke$^{\cite{9Pommerenke,10Pommerenke}}$ first introduced the Hankel determinant
$$\mathcal{H}_j(n)=\begin{vmatrix}a_n&a_{n+1}&\cdots&a_{n+j-1}\\a_{n+1}&a_{n+2}&\cdots&a_{n+j}
\\\vdots&\vdots&\ddots&\vdots
\\a_{n+j-1}&a_{n+j}&\cdots&a_{n+2j-2}\end{vmatrix}$$
for $j, n\in\mathbb{N}$. Clearly,
$$\mathcal{H}_2(1)=\begin{vmatrix}1&a_2\\a_2&a_3\end{vmatrix}=a_3-a_2^2.$$

In recent years, many researchers started to study $q$-calculus, which has been widely applied in many fields. The $q$-calculus, which is grounded in quantum groups and Lie algebras, provides a framework for analyzing symmetries and conservation laws in the microscopic realm. By introducing the parameter $q$ (usually taken as a real number), researchers can construct a series of mathematical models paralleled to classical calculus. Below, we review some basic concepts of $q$-calculus.

For $q\in (0,1)$, the $q$-number and $q$-factorial$^{\cite{3Ma}}$ are defined respectively as
\begin{align}\nonumber
[\lambda]_q=\begin{cases}\frac{1-q^\lambda}{1-q},\ &(\lambda\in\mathbb{C}),
\\\sum_{k=0}^{n-1}q^k=1+q+q^2+...+q^{n-1},\ &(\lambda=n\in\mathbb{N})\end{cases}
\end{align}
and
\begin{align}\nonumber
[n]_q!=\begin{cases}1,\ &(n=0),
\\\prod^{n}_{k=1}[k]_{q},\ &(n\in\mathbb{N}). \end{cases}
\end{align}

The $q$-derivative operator or $q$-differential operator $D_{q} $ $^{\cite{4Jackson,5Jackson}}$ for a function $f(z)$ is defined as
\begin{align}\label{e1.2}
(D_{q} f)(z)=\left\{\begin{matrix}\frac{f(z)-f(qz)}{(1-q)z},\  &(z\ne 0),
 \\{f}'(0),\ &(z=0) ,
\end{matrix}\right.
\end{align}
and consequently,
\begin{align}
\lim_{q \to 1^{-} } (D_{q}f )(z)=\lim_{q \to 1^{-} } \frac{f(qz)-f(z)}{(q-1)z} ={f}' (z). \nonumber
\end{align}
Furthermore, combining (\ref{e1.1}) and (\ref{e1.2}), we derive that
\begin{align}
 (D_{q}f )(z)=1+\sum_{n=2}^{\infty } \left [ n \right ] _{q} a_{n} z^{n-1} .\nonumber
\end{align}

Early on, the $q$-derivative operator$^{\cite{18F}}$ pioneered by Jackson et al. stands as a crucial instrument in quantum group theory and special function theory. 
Subsequently, the $(p,q)$-derivative operator$^{\cite{15Srivastava}}$, grounded in dual-parameter regulation, further expands this structural paradigm. 
Later the advancements led to the $(\alpha,\beta,\gamma,)$-derivative operator$^{\cite{16Vahid,17Kanas,20Long}}$, which employs a triparametric system to 
enhance adaptability of complex geometric configurations and maintain algebraic consistency. Building upon these foundational works, 
our study innovatively unifies these representative classes of specialize derivative operators into g-derivative operator. Notably, 
this paper systematically introduces a novel class of derivative operators embedded within a multidimensional parameter space. 
This achievement not only refines the differential operator framework in function theory but also establishes a unified mathematical language for fractional calculus and quantum group theory. 
Next, we introduce a new class of derivative operators.
\begin{definition}
Let $g(z)=z+\sum_{n=2}^{\infty}g_{n}z^{n}$ be an analytic function in $ \mathbb{A} $. For $f\in \mathbb{A}$, the $g$-derivative operator $D_{g}$ is defined as
\begin{align}\nonumber
D_{g}f(z)=&\frac{1}{z}\{f(z)\ast g(z)\}\\\nonumber
=&\frac{1}{z}\left\{\left(z+\sum_{n=2}^{\infty}a_{n}z^{n}\right) \ast \left(z+ \sum_{n=2}^{\infty}g_{n}z^{n}\right)\right\}\\\nonumber
=&\frac{1}{z}\left\{z+ \sum_{n=2}^{\infty}[n]_{g}a_{n}z^{n}\right\}\nonumber,
\end{align}
where $[n]_{g}=g_{n}$ and $\ast$ is a convolution.

The following are some specific cases of the $g$-derivative operator.
\end{definition}
\begin{remark}
The function $g(z)$ is defined as
\begin{align}
g(z)=\frac{2(1-\gamma)z}{(1-\alpha z)(1-\beta z)},\nonumber
\end{align}
where $\alpha\in[-1,1], \beta\in[-1,1], \alpha\beta\neq\pm1$ and $\gamma\in[0,1]$. If $f\in \mathbb{A}$  and $\alpha \neq \beta$, then the $g$-derivative operator is given by
\begin{align}\nonumber
D_{g}f(z)=: D_{\alpha,\beta,\gamma}f(z)=&\frac{1}{z}\{f(z)\ast g(z)\}\\\nonumber
=&\frac{1}{z}\left\{\left(z+\sum_{n=2}^{\infty}a_{n}z^{n}\right) \ast \left(z+ \sum_{n=1}^{\infty}2(1-\gamma)\left(\frac{\alpha^{n}-\beta^{n}}{\alpha-\beta}\right)z^{n}\right)\right\}\\\nonumber
=&\frac{1}{z}\left\{z+ \sum_{n=2}^{\infty}[n]_{\alpha,\beta,\gamma}a_{n}z^{n}\right\},\nonumber
\end{align}
where
\begin{align}
[n]_{\alpha,\beta,\gamma}=2(1-\gamma)\left(\frac{\alpha^{n}-\beta^{n}}{\alpha-\beta}\right),\ \ \ (n=2,3,...)\nonumber.
\end{align}
Here the different cases are obtained when the parameters take different values.

(1) If $\gamma=1/2, \alpha=p, \beta=q$, where $p,q\in[-1,1], pq\neq\pm1$, and $p \neq q$, then $g(z)=\frac{z}{(1-pz)(1-qz)}$, and the $g$-derivative operator degenerates into the $(p, q)$-derivative operator.

(2) If $\gamma=1/2, \alpha=1, \beta=q$, where $q\in(0,1)$, then $g(z)=\frac{z}{(1-qz)(1-z)}$, and the $g$-derivative operator reduces to the $q$-derivative operator.

(3) Let $\gamma=1/2, \alpha=1, \beta=q$, where $q\in(0,1)$. If $q\rightarrow1^{-}$, then $g(z)=\frac{z}{(1-z)^{2}}$, and the corresponding $g$-derivative operator is exactly the classical derivative operator.
\end{remark}

According to the statement above, we define the class of $g$-starlike functions and the class of $g$-convex functions in the unit disk by using the $g$-derivative operator. Then we investigate several properties of these new function subclasses, such as the estimations for the first few coefficients, Fekete-Szeg\"{o} inequalities, Toeplitz determinants and Hankel determinants. Additionally, we provide some remarks, examples and corollaries.

\section{The Class $\mathcal{S}^{\ast }_{g}(\phi)$ Of Starlike Functions }

Let $\mathcal{P}$ be the class of all analytic and univalent functions $p(z)$ of the form:
\begin{equation}\label{e2.1}
p(z) = 1 + \sum_{n=1}^{\infty} c_n z^n, \quad (z \in \mathbb{U})
\end{equation}
satisfying $\Re[p(z)] > 0$ and $p(0) = 1$.

Now we first present several relevant lemmas.
\begin{lemma} $^{\cite{11Kanas}}$
If $p \in \mathcal{P}$, then the sharp estimate 
\begin{equation*}
\left|c_{n}\right| \leq 2, \quad (n \in \mathbb{N})
\end{equation*}
holds. In particular, the equality is achieved for the function
\begin{equation*}
p(z) = \frac{1+z}{1-z} = 1 + \sum_{k=1}^{\infty} 2 z^k.
\end{equation*}
\end{lemma}

\begin{lemma}$^{\cite{12Hayman}}$
If $p \in \mathcal{P}$, then 
\begin{equation*}
\left|c_{2} - \mu c_{1}^{2}\right| \leq 2 \max\{1, |2\mu - 1|\}
\end{equation*}
for any complex number $\mu$, and the sharp of this result is true for the functions
\begin{equation*}
p_1(z) = \frac{1+z}{1-z} \quad \text{and} \quad p_2(z) = \frac{1+z^2}{1-z^2}, \quad (z \in \mathbb{U}).
\end{equation*}
\end{lemma}

\begin{lemma}$^{\cite{13Singh}}$
Suppose that $p \in \mathcal{P}$ and $\mu \in \mathbb{R}$. Then
\begin{equation*}
\left|c_2 - \mu c_1^2\right| \leq
\begin{cases}
-4\mu + 2, & \mu \leq 0, \\
2, &  0 \leq \mu \leq 1, \\
4\mu - 2, &  \mu \geq 1,
\end{cases}
\end{equation*}
and this result is sharp.
\end{lemma}

\begin{lemma}$^{\cite{14Pommerenke}}$
If $p \in \mathcal{P}$, then
\begin{align*}
\left|c_{n+k} - \mu c_n c_k\right| &\leq 2, \ \   0 \leq \mu \leq 1, \\
\left|c_2 - \frac{1}{2} c_1^2\right| &\leq 2 - \frac{\left|c_1\right|^2}{2}, \\
\left|c_n c_m - c_\ell c_k\right| &\leq 4, \ \   n+m = \ell+k, \\
\left|c_{n+2k} - \mu c_n c_k^2\right| &\leq 2(1+2\mu), \ \   \mu \in \mathbb{R}.
\end{align*}
\end{lemma}

Herein, we are ready to introduce the class $\mathcal{S}^{\ast }_{g}(\phi)$ of starlike functions.

\begin{definition}
If the following subordination relation
\begin{align}
\frac{zD_{g}f(z)}{f(z)}\prec \phi(z),\ \  (z\in\mathbb{U})\nonumber
\end{align}
is satisfied for $f\in \mathbb{A}$, where $g\in \mathbb{A}$, then $f$ is said to belong to $f\in\mathcal{S}^{\ast }_{g}(\phi)$.

For simplicity, the function $f$ is said to belong to $f\in\mathcal{S}^{\ast }_{\alpha, \beta, \gamma}(\phi)$ when $g(z)=\frac{2(1-\gamma)z}{(1-\alpha z)(1-\beta z)}$ in Remark 1.2. Similarly,  $f\in\mathcal{S}^{\ast }_{p,q}(\phi)$, $f\in\mathcal{S}^{\ast }_{q}(\phi)$ and $f\in\mathcal{S}^{\ast }_{1^{-}}(\phi)$ respectively when $g(z)$ corresponds to the cases (1), (2), and (3) in this Remark.

\end{definition}

Now, an example related to the function class $\mathcal{S}^{\ast }_{g}(\phi)$ is subsequently provided. 
\begin{example}
A function $f\in\mathcal{S}^{\ast }_{1^{-}}(\phi)$ if and only if there exists an analytic function $h\prec\phi(z)$ such that
\begin{align}\nonumber
f(z)=z\exp\int^{z}_{0}\frac{h(t)-1}{t}dt,\ \ \ z\in\mathbb{U}.
\end{align}
This integral representation supplies many examples of functions in the class $\mathcal{S}^{\ast }_{1^{-}}(\phi)$. Let
\begin{align}
h(t)=\phi(t^{n})=(1+t^{n})^{\lambda},\nonumber
\end{align}
where $n\in\mathbb{N}$, and $t\in\mathbb{U}$.

Through the straightforward computation, the function (as shown in Figure 1)
\begin{align}
\varphi(z)=&z\exp\int^{z}_{0}\frac{(1+t^{n})^{\lambda}-1}{t}dt\nonumber
\\=&z+\frac{\lambda}{n}z^{n+1}+\frac{\lambda^{2}(n+2)-n\lambda}{4n^{2}}z^{2n+1}\nonumber
\\+&\frac{\lambda((2n^{2}+9n+6)\lambda^{2}-(6n^{2}+9n)\lambda+4n^{2})}{36n^{3}}z^{3n+1}+\cdot\cdot\cdot\nonumber
\end{align}
is an extremal function in $\mathcal{S}^{\ast }_{1^{-}}(\phi)$.
\begin{figure}[!ht]
  \centering
  \includegraphics[width=6cm]{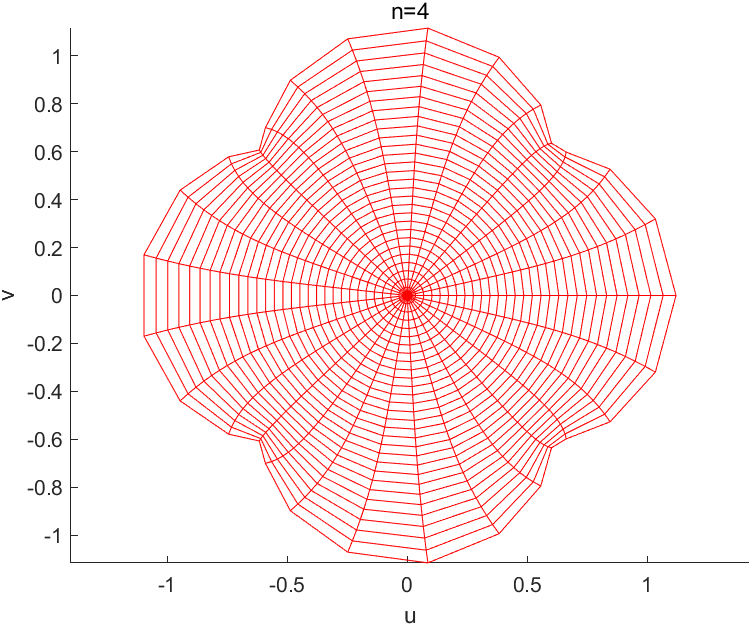}
  \includegraphics[width=6cm]{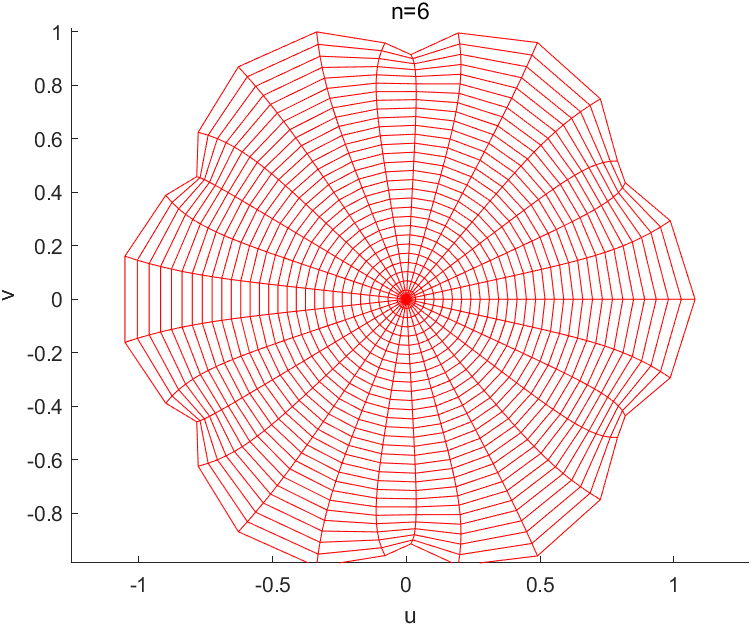}
  \caption{The image of $\mathbb{U}$ under $\varphi(z)$ for $n=4, 6$ and $\lambda=\frac{1}{2}$.}
  \label{4}
\end{figure}
\end{example}
Below, we study the coefficients of the function class $\mathcal{S}^{\ast }_{g}(\phi)$ and Fekete-Szeg\"{o} inequality. First, we estimate the initial coefficients of the function class $\mathcal{S}^{\ast }_{g}(\phi)$.

\begin{theorem}
Suppose that $f$ is given by (\ref{e1.1}). If $f$ belongs to the function class $\mathcal{S}^{\ast }_{g}(\phi)$, then
\begin{align}
\mid a_{2}\mid \leq \frac{B_{1}}{[2]_{g}-1}\nonumber,
\end{align}
\begin{align}\nonumber
\mid a_{3}\mid \leq\frac{B_{1}}{[3]_{g}-1}\max\left\{1, \left|\frac{B_{2}}{B_{1}}+\frac{B_{1}}{[2]_{g}-1}\right|\right\}.\nonumber
\end{align}
\end{theorem}

\begin{proof}
Assume that $f\in\mathcal{S}^{\ast }_{g}(\phi)$. Then, there exists a Schwarz function $u(z)$ analytic in $\mathbb{U}$ with $u(0) = 0$ and $|u(z)| \leq 1$ such that
\begin{align}\label{e2.2}
\frac{zD_{g}f(z)}{f(z)}=\phi(u(z)).
\end{align}
Let $p \in \mathcal{P}$ satisfy
\begin{equation*}
p(z) = \frac{1+u(z)}{1-u(z)} .
\end{equation*}
Through the straightforward computation, from (2.1) we obtain
\begin{align*}
u(z) = \frac{p(z) - 1}{p(z) + 1} &= \frac{c_1}{2}z + \left(\frac{c_2}{2} - \frac{c_1^2}{4}\right)z^2 + \left(\frac{c_3}{2} - \frac{c_2c_1}{2} + \frac{c_1^3}{8}\right)z^3 \\
&+ \left(\frac{c_4}{2} - \frac{c_3c_1}{2} + \frac{3c_2c_1^2}{8} - \frac{c_1^4}{16}\right)z^4 + \ldots, \quad (z \in U).
\end{align*}

Assume that $\phi$ is an analytic function in the open unit disk with positive real part, satisfying $\phi(0)=1, \phi'(0)>0$ and $\phi(\mathbb{D})$ is symmetric with respect to the real axis. Then, it has the following representation
\begin{align}
\phi(z)=1+B_{1}z+B_{2}z^{2}+B_{3}z^{3}+\cdot\cdot\cdot\ \ (B_{1}>0)\nonumber.
\end{align}
Hence
\begin{align}\nonumber\label{e2.3}
\phi(u(z))&=\phi\left(\frac{p(z) - 1}{p(z) + 1}\right)\\
&=1+\frac{c_{1}B_{1}}{2}z+\left[\left(\frac{c_{2}}{2}-\frac{c_{1}^{2}}{4}\right)B_{1}+\frac{c_{1}^{2}}{4}B_{2}\right]z^{2}+\cdot\cdot\cdot.
\end{align}
Since $f\in \mathbb{A} $, through the simple computation we have
\begin{align}\label{e2.4}
\frac{zD_{g}f(z)}{f(z)}=1+\left([2]_{g}-1\right)a_{2}z+\left(([3]_{g}-1)a_{3}-([2]_{g}-1)a_{2}^{2}\right)z^{2}+\cdot\cdot\cdot.
\end{align}
By combining (\ref{e2.2}), (\ref{e2.3}) and (\ref{e2.4}), we obtain
\begin{align}
\left([2]_{g}-1\right)a_{2} &= \frac{c_{1}B_{1}}{2}, \\
([3]_{g}-1)a_{3}-([2]_{g}-1)a_{2}^{2} &= \left(\frac{c_{2}}{2}-\frac{c_{1}^{2}}{4}\right)B_{1}+\frac{c_{1}^{2}}{4}B_{2}.
\end{align}
From (2.5) and (2.6), it follows that
\begin{align}
a_{2} &= \frac{c_1B_{1}}{2([2]_{g}-1)}, \\
a_{3} &= \frac{B_{1}}{2([3]_{g}-1)}\left[c_{2}-\frac{1}{2}\left(1-\frac{B_{2}}{B_{1}}-\frac{B_{1}}{[2]_{g}-1}\right)c_{1}^{2}\right].
\end{align}
Applying Lemma 2.1 and Lemma 2.2 to (2.7) and (2.8), respectively, we obtain
\begin{align}
\mid a_{2}\mid \leq \frac{B_{1}}{[2]_{g}-1}\nonumber
\end{align}
and
\begin{align}\nonumber
\mid a_{3}\mid \leq\frac{B_{1}}{[3]_{g}-1}\max\left\{1, \left|\frac{B_{2}}{B_{1}}+\frac{B_{1}}{[2]_{g}-1}\right|\right\}.\nonumber
\end{align}
Thus, Theorem 2.7 is proved.
\end{proof}
Next, by substituting different $g(z)$  from Remark 1.2 into Theorem 2.7, we obtain the following corollaries.

\begin{corollary}
Let $f$ be given by (1.1). If $f\in\mathcal{S}^{\ast }_{q}(\phi)$ , then
\begin{align}
\mid a_{2}\mid \leq \frac{B_{1}}{q}\nonumber,\ \
\mid a_{3}\mid \leq\frac{B_{1}}{q(1+q)}\max\left\{1, \left|\frac{B_{2}}{B_{1}}+\frac{B_{1}}{q}\right|\right\}.\nonumber
\end{align}
When $q\rightarrow1^{-}$,
\begin{align}
\mid a_{2}\mid \leq B_{1}\nonumber,\ \
\mid a_{3}\mid \leq\frac{B_{1}}{2}\max\left\{1, \left|\frac{B_{2}}{B_{1}}+B_{1}\right|\right\}.\nonumber
\end{align}
\end{corollary}

\begin{corollary}
Let $f$ be given by (1.1). If $f\in\mathcal{S}^{\ast }_{\alpha, \beta, \gamma}(\phi)$, then
\begin{align}
\mid a_{2}\mid \leq \frac{B_{1}}{[2]_{\alpha, \beta, \gamma}-1}\nonumber,\ \
\mid a_{3}\mid \leq\frac{B_{1}}{[3]_{\alpha, \beta, \gamma}-1}\max\left\{1, \left|\frac{B_{2}}{B_{1}}+\frac{B_{1}}{[2]_{\alpha, \beta, \gamma}-1}\right|\right\}.\nonumber
\end{align}
\end{corollary}

\begin{theorem}
Assume that $f$ is given by (\ref{e1.1}). If $f$ belongs to the function class $\mathcal{S}^{\ast }_{g}(\phi)$, then for a complex number $\mu$,
\begin{align}\nonumber
\mid a_3-\mu a_2^2\mid\leq \frac{B_{1}}{[3]_{g}-1}\max\left\{1,\left|\frac{B_{2}}{B_{1}}+
\frac{B_{1}}{[2]_{g}-1}\left(1-\frac{[3]_{g}-1}{[2]_{g}-1}\mu\right)\right|\right\}.
\end{align}
\end{theorem}

\begin{proof}
From (2.7) and (2.8), we have
\begin{align}\nonumber
a_{2}^{2} &= \frac{c_{1}^{2}B_{1}^{2}}{4([2]_{g}-1)^{2}}, \\\nonumber
a_{3} &= \frac{B_{1}}{2([3]_{g}-1)}\left[c_{2}-\frac{1}{2}\left(1-\frac{B_{2}}{B_{1}}-\frac{B_{1}}{[2]_{g}-1}\right)c_{1}^{2}\right].
\end{align}
Thus
\begin{align}\nonumber\label{e2.9}
a_{3}-\mu a_{2}^{2}&=\frac{B_{1}}{2([3]_{g}-1)}\left[c_{2}-\frac{1}{2}\left(1-\frac{B_{2}}{B_{1}}-\frac{B_{1}}{[2]_{g}-1}\right)c_{1}^{2}\right]-\frac{\mu c_{1}^{2}B_{1}^{2}}{4([2]_{g}-1)^{2}}\\
&=\frac{B_{1}}{2([3]_{g}-1)}(c_{2}-\kappa c_{1}^{2}),
\end{align}
where
\begin{align}\nonumber
\kappa=\frac{1}{2}\left(1-\frac{B_{2}}{B_{1}}-
\frac{B_{1}}{[2]_{g}-1}\left(1-\frac{[3]_{g}-1}{[2]_{g}-1}\mu\right)\right).
\end{align}
Applying Lemma 2.2 to (2.9), we obtain
\begin{align}\nonumber
\mid a_3-\mu a_2^2\mid\leq \frac{B_{1}}{[3]_{g}-1}\max\left\{1,\left|\frac{B_{2}}{B_{1}}+
\frac{B_{1}}{[2]_{g}-1}\left(1-\frac{[3]_{g}-1}{[2]_{g}-1}\mu\right)\right|\right\}.
\end{align}
Thus, Theorem 2.10 is proved.
\end{proof}
Similarly, substituting different $g(z)$ from Remark 1.2 into Theorem 2.10 yields the following corollaries.

\begin{corollary}
Let $f$ be given by (1.1). If $f\in\mathcal{S}^{\ast }_{q}(\phi)$ , then for $\mu \in \mathbb{C}$,
\begin{align}\nonumber
\mid a_3-\mu a_2^2\mid\leq \frac{B_{1}}{q(1+q)}\max\left\{1,\left|\frac{B_{2}}{B_{1}}+
\frac{B_{1}}{q}\left(1-(1+q)\mu\right)\right|\right\}.
\end{align}
When $q\rightarrow1^{-}$,
\begin{align}\nonumber
\mid a_3-\mu a_2^2\mid\leq \frac{B_{1}}{2}\max\left\{1,\left|\frac{B_{2}}{B_{1}}+
B_{1}\left(1-2\mu\right)\right|\right\}.
\end{align}
\end{corollary}

\begin{corollary}
Let $f$ be given by (1.1). If $f\in\mathcal{S}^{\ast }_{\alpha, \beta, \gamma}(\phi)$, then for $\mu \in \mathbb{C}$,
\begin{align}\nonumber
\mid a_3-\mu a_2^2\mid\leq \frac{B_{1}}{[3]_{\alpha, \beta, \gamma}-1}\max\left\{1,\left|\frac{B_{2}}{B_{1}}+
\frac{B_{1}}{[2]_{\alpha, \beta, \gamma}-1}\left(1-\frac{[3]_{\alpha, \beta, \gamma}-1}{[2]_{\alpha, \beta, \gamma}-1}\mu\right)\right|\right\}.
\end{align}
\end{corollary}
If we take $\mu \in \mathbb{R}$, then according to the proof of Theorem 2.10 and Lemma 2.3, the Fekete-Szeg\"{o} inequality for the function class $\mathcal{S}^{\ast }_{g}(\phi)$ can be estimated.
\begin{theorem}
Let $f$ be given by (1.1) with $\mu \in \mathbb{R}$. If $f \in \mathcal{S}^{\ast }_{g}(\phi)$, then
\begin{equation*}
\left|a_{3} - \mu a_{2}^{2}\right| \leq \frac{B_{1}}{[3]_{g}-1}
\begin{cases}
-2\kappa + 1, & \mu \leq \kappa_{1}, \\
1, &  \kappa_{1} \leq \mu \leq \kappa_{2}, \\
2\kappa - 1, &  \mu \geq \kappa_{2},
\end{cases}
\end{equation*}
where
\begin{align}\nonumber
\kappa_{1}=\frac{B_{1}^{2}([2]_{g}-1)+(B_{2}-B_{1})([2]_{g}-1)^{2}}{B_{1}^{2}([3]_{g}-1)},
\end{align}
\begin{align}\nonumber
\kappa_{2}=\frac{B_{1}^{2}([2]_{g}-1)+(B_{2}+B_{1})([2]_{g}-1)^{2}}{B_{1}^{2}([3]_{g}-1)}.
\end{align}
\end{theorem}
By inserting $g(z)$ subject to diverse scenarios outlined in Remark 1.2 into Theorem 2.13, we derive the subsequent corollaries.

\begin{corollary}
Let $f$ be given by (1.1). If $f\in\mathcal{S}^{\ast }_{q}(\phi)$ , then for $\mu \in \mathbb{R}$,
\begin{equation*}
\left|a_{3} - \mu a_{2}^{2}\right| \leq \frac{B_{1}}{2(q+q^{2})}
\begin{cases}
-4\kappa + 2, & \mu \leq \kappa_{1}, \\
2, &  \kappa_{1} \leq \mu \leq \kappa_{2}, \\
4\kappa - 2, &  \mu \geq \kappa_{2},
\end{cases}
\end{equation*}
where
\begin{align}\nonumber
\kappa_{1}=\frac{B_{1}^{2}+(B_{2}-B_{1})q}{B_{1}^{2}(1+q)},
\end{align}
\begin{align}\nonumber
\kappa_{2}=\frac{B_{1}^{2}+(B_{2}+B_{1})q}{B_{1}^{2}(1+q)}.
\end{align}
When$q\rightarrow1^{-}$,
\begin{equation*}
\left|a_{3} - \mu a_{2}^{2}\right| \leq \frac{B_{1}}{2}
\begin{cases}
-2\kappa + 1, & \mu \leq \kappa_{1}, \\
1, &  \kappa_{1} \leq \mu \leq \kappa_{2}, \\
2\kappa - 1, &  \mu \geq \kappa_{2},
\end{cases}
\end{equation*}
where
\begin{align}\nonumber
\kappa_{1}=\frac{B_{1}^{2}+B_{2}-B_{1}}{2B_{1}^{2}},\ \
\kappa_{2}=\frac{B_{1}^{2}+B_{2}+B_{1}}{2B_{1}^{2}}.
\end{align}
\end{corollary}

\begin{corollary}
Let $f$ be given by (1.1). If $f\in\mathcal{S}^{\ast }_{\alpha, \beta, \gamma}(\phi)$, then for $\mu \in \mathbb{R}$,
\begin{equation*}
\left|a_{3} - \mu a_{2}^{2}\right| \leq \frac{B_{1}}{2([3]_{\alpha, \beta, \gamma}-1)}
\begin{cases}
-4\kappa + 2, & \mu \leq \kappa_{1}, \\
2, &  \kappa_{1} \leq \mu \leq \kappa_{2}, \\
4\kappa - 2, &  \mu \geq \kappa_{2},
\end{cases}
\end{equation*}
where
\begin{align}\nonumber
\kappa_{1}=\frac{B_{1}^{2}([2]_{\alpha, \beta, \gamma}-1)+(B_{2}-B_{1})([2]_{\alpha, \beta, \gamma}-1)^{2}}{B_{1}^{2}([3]_{\alpha, \beta, \gamma}-1)},
\end{align}
\begin{align}\nonumber
\kappa_{2}=\frac{B_{1}^{2}([2]_{\alpha, \beta, \gamma}-1)+(B_{2}+B_{1})([2]_{\alpha, \beta, \gamma}-1)^{2}}{B_{1}^{2}([3]_{\alpha, \beta, \gamma}-1)}.
\end{align}
\end{corollary}
Next, we proceed to estimate the Toeplitz determinant for the function class $\mathcal{S}^{\ast }_{g}(\phi)$ by starting with the estimation of $\mathcal{T}_2(2)$.

\begin{theorem}
Assume that $f$ is given by (1.1). If $f\in\mathcal{S}^{\ast }_{g}(\phi)$, then

(1) If $|\frac{B_{2}}{B_{1}}+\frac{B_{1}}{[2]_{g}-1}|\leq1$, then
\begin{align}\nonumber
\mid \mathcal{T}_2(2)\mid=\mid a_{2}^2-a_{3}^{2}\mid=
\frac{B_{1}^{2}}{([2]_{g}-1)^{2}}-\frac{B_{1}^{2}}{([3]_{g}-1)^{2}}.
\end{align}

(2) If $|\frac{B_{2}}{B_{1}}+\frac{B_{1}}{[2]_{g}-1}|>1$, then
\begin{align}\nonumber
\mid \mathcal{T}_2(2)\mid=\mid a_{2}^2-a_{3}^{2}\mid\leq
\frac{B_{1}^{2}}{([2]_{g}-1)^{2}}+\frac{B_{1}^{2}}{([3]_{g}-1)^{2}}
\times\left|\frac{B_{2}}{B_{1}}+\frac{B_{1}}{[2]_{g}-1}\right|^{2}.
\end{align}
\end{theorem}
By substituting different $g(z)$ from Remark 1.2 into Theorem 2.16, we acquire the following corollaries.
\begin{corollary}
If $f$ is given by (\ref{e1.1}) and $f\in\mathcal{S}^{\ast }_{q}(\phi)$ , then

(1) If $|\frac{B_{2}}{B_{1}}+\frac{B_{1}}{q}|\leq1$, then
\begin{align}\nonumber
\mid \mathcal{T}_2(2)\mid=\mid a_{2}^2-a_{3}^{2}\mid=
\frac{B_{1}^{2}(2+q)}{q(1+q)^{2}}.
\end{align}

(2) If $|\frac{B_{2}}{B_{1}}+\frac{B_{1}}{q}|>1$, then
\begin{align}\nonumber
\mid \mathcal{T}_2(2)\mid=\mid a_{2}^2-a_{3}^{2}\mid\leq
\frac{B_{1}^{2}}{q^{2}}\left(1+\frac{1}{(1+q)^{2}}
\times\left|\frac{B_{2}}{B_{1}}+\frac{B_{1}}{q}\right|^{2}\right).
\end{align}
Let $q\rightarrow1^{-}$,

(1) If $|\frac{B_{2}}{B_{1}}+B_{1}|\leq1$, then
\begin{align}\nonumber
\mid \mathcal{T}_2(2)\mid=\mid a_{2}^2-a_{3}^{2}\mid\leq
\frac{3B_{1}^{2}}{4}.
\end{align}

(2) If $|\frac{B_{2}}{B_{1}}+B_{1}|>1$, then
\begin{align}\nonumber
\mid \mathcal{T}_2(2)\mid=\mid a_{2}^2-a_{3}^{2}\mid\leq
B_{1}^{2}+\frac{B_{1}^{2}}{4}
\times\left|\frac{B_{2}}{B_{1}}+B_{1}\right|^{2}.
\end{align}
\end{corollary}

\begin{corollary}
If $f$ is given by (\ref{e1.1}) and $f\in\mathcal{S}^{\ast }_{\alpha, \beta, \gamma}(\phi)$ , then

(1) If $|\frac{B_{2}}{B_{1}}+\frac{B_{1}}{[2]_{\alpha, \beta, \gamma}-1}|\leq1$, then
\begin{align}\nonumber
\mid \mathcal{T}_2(2)\mid=\mid a_{2}^2-a_{3}^{2}\mid=
\frac{B_{1}^{2}}{([2]_{\alpha, \beta, \gamma}-1)^{2}}-\frac{B_{1}^{2}}{([3]_{\alpha, \beta, \gamma}-1)^{2}}.
\end{align}

(2) If $|\frac{B_{2}}{B_{1}}+\frac{B_{1}}{[2]_{\alpha, \beta, \gamma}-1}|>1$, then
\begin{align}\nonumber
\mid \mathcal{T}_2(2)\mid=\mid a_{2}^2-a_{3}^{2}\mid\leq
\frac{B_{1}^{2}}{([2]_{\alpha, \beta, \gamma}-1)^{2}}+\frac{B_{1}^{2}}{([3]_{\alpha, \beta, \gamma}-1)^{2}}
\times\left|\frac{B_{2}}{B_{1}}+\frac{B_{1}}{[2]_{\alpha, \beta, \gamma}-1}\right|^{2}.
\end{align}
\end{corollary}
Based on the proofs of Theorem 2.10 and Lemma 2.2, we consider the symmetric Toeplitz determinant $\mathcal{T}_3(1)$ of $\mathcal{S}^{\ast }_{g}(\phi)$ and obtain the following theorem.
\begin{theorem}
Suppose that $f$ is given by (\ref{e1.1}).  If $f\in\mathcal{S}^{\ast }_{g}(\phi)$, then
\begin{align}\nonumber
\mid \mathcal{T}_3(1)\mid &\leq \left(1+\frac{B_{1}}{[3]_{g}-1}\max\left\{1,\left|\frac{B_{2}}{B_{1}}+
\frac{B_{1}}{[2]_{g}-1}\left(1-2\frac{[3]_{g}-1}{[2]_{g}-1}\right)\right|\right\}\right)\\\nonumber
&\times\left(1+\frac{B_{1}}{[3]_{g}-1}\max\left\{1, \left|\frac{B_{2}}{B_{1}}+\frac{B_{1}}{[2]_{g}-1}\right|\right\}\right).
\end{align}
\end{theorem}
\begin{proof}
By the definition of the symmetric Toeplitz determinant, we have
\begin{align}
\mid \mathcal{T}_3(1)\mid=|(1+a_{3}-2a_{2}^2)(1-a_{3})|\le (1+|a_{3}-2a_{2}^2|)(1+|a_{3}|).
\end{align}
Applying Theorem 2.10 and Lemma 2.2 to the inequality (2.10), we obtain
\begin{align}\nonumber
\mid \mathcal{T}_3(1)\mid &\leq \left(1+\frac{B_{1}}{[3]_{g}-1}\max\left\{1,\left|\frac{B_{2}}{B_{1}}+
\frac{B_{1}}{[2]_{g}-1}\left(1-2\frac{[3]_{g}-1}{[2]_{g}-1}\right)\right|\right\}\right)\\\nonumber
&\times\left(1+\frac{B_{1}}{[3]_{g}-1}\max\left\{1, \left|\frac{B_{2}}{B_{1}}+\frac{B_{1}}{[2]_{g}-1}\right|\right\}\right).
\end{align}
Thus, Theorem 2.19 holds.
\end{proof}
Similarly, by taking different $g(z)$ from Remark 1.2 in Theorem 2.19, the following corollaries can be derived.

\begin{corollary}
If $f$ is given by (\ref{e1.1}), and if $f\in\mathcal{S}^{\ast }_{q}(\phi)$ , then
\begin{align}\nonumber
\mid \mathcal{T}_3(1)\mid &\leq \left(1+\frac{B_{1}}{q(1+q)}\max\left\{1,\left|\frac{B_{2}}{B_{1}}-
\frac{B_{1}}{q}\left(1+2q\right)\right|\right\}\right)\\\nonumber
&\times\left(1+\frac{B_{1}}{q(1+q)}\max\left\{1, \left|\frac{B_{2}}{B_{1}}+\frac{B_{1}}{q}\right|\right\}\right).
\end{align}
When $q\rightarrow1^{-}$,
\begin{align}\nonumber
\mid \mathcal{T}_3(1)\mid \leq \left(1+\frac{B_{1}}{2}\max\left\{1,\left|\frac{B_{2}}{B_{1}}-3
B_{1}\right|\right\}\right)
\times\left(1+\frac{B_{1}}{2}\max\left\{1, \left|\frac{B_{2}}{B_{1}}+B_{1}\right|\right\}\right).
\end{align}
\end{corollary}

\begin{corollary}
If $f$ is given by (\ref{e1.1}),  and if $f\in\mathcal{S}^{\ast }_{\alpha, \beta, \gamma}(\phi)$ , then
\begin{align}\nonumber
\mid \mathcal{T}_3(1)\mid &\leq \left(1+\frac{B_{1}}{[3]_{\alpha, \beta, \gamma}-1}\max\left\{1,\left|\frac{B_{2}}{B_{1}}+
\frac{B_{1}}{[2]_{\alpha, \beta, \gamma}-1}\left(1-2\frac{[3]_{\alpha, \beta, \gamma}-1}{[2]_{\alpha, \beta, \gamma}-1}\right)\right|\right\}\right)\\\nonumber
&\times\left(1+\frac{B_{1}}{[3]_{\alpha, \beta, \gamma}-1}\max\left\{1, \left|\frac{B_{2}}{B_{1}}+\frac{B_{1}}{[2]_{\alpha, \beta, \gamma}-1}\right|\right\}\right).
\end{align}
\end{corollary}
Similarly, based on the proof of Theorem 2.12 and Lemma 23, we consider the symmetric determinant $\mathcal{T}_3(1)$ of $\mathcal{S}^{\ast }_{g}(\phi)$ and obtain the following theorem.

\begin{theorem}
Assume that $f$ is given by (\ref{e1.1}). If $f$ belongs to the function class $\mathcal{S}^{\ast }_{g}(\phi)$, then
\begin{align}\nonumber
\mid \mathcal{T}_3(1)\mid &\leq \left(1+\frac{B_{1}}{[3]_{g}-1}\times\left(\frac{B_{1}}{[2]_{g}-1}\left(2\frac{[3]_{g}-1}{[2]_{g}-1}-1\right)-\frac{B_{2}}{B_{1}}\right)\right)\\\nonumber
&\times\left(1+\frac{B_{1}}{[3]_{g}-1}\max\left\{1, \left|\frac{B_{2}}{B_{1}}+\frac{B_{1}}{[2]_{g}-1}\right|\right\}\right).
\end{align}
\end{theorem}

\begin{proof}
From the determinant $\mathcal{T}_3(1)$ and Theorem 2.12, we have
\begin{align}\nonumber
\mid a_{3}-2a_{2}^{2}\mid & \leq \frac{B_{1}}{2([3]_{g}-1)}(4\kappa - 2) \\\nonumber
&\leq\frac{B_{1}}{[3]_{g}-1}\times\left(\frac{B_{1}}{[2]_{g}-1}\left(2\frac{[3]_{g}-1}{[2]_{g}-1}-1\right)-\frac{B_{2}}{B_{1}}\right).
\end{align}
Therefore
\begin{align}\nonumber
\mid \mathcal{T}_3(1)\mid &\leq \left(1+\frac{B_{1}}{[3]_{g}-1}\times\left(\frac{B_{1}}{[2]_{g}-1}\left(2\frac{[3]_{g}-1}{[2]_{g}-1}-1\right)-\frac{B_{2}}{B_{1}}\right)\right)\\\nonumber
&\times\left(1+\frac{B_{1}}{[3]_{g}-1}\max\left\{1, \left|\frac{B_{2}}{B_{1}}+\frac{B_{1}}{[2]_{g}-1}\right|\right\}\right).
\end{align}
Thus, Theorem 2.22 holds.
\end{proof}
Similarly, by applying different $g(z)$ from Remark 1.2 to Theorem 2.22, we obtain the following corollaries.
\begin{corollary}
If $f$ is given by (\ref{e1.1}),  and if $f\in\mathcal{S}^{\ast }_{q}(\phi)$ , then
\begin{align}\nonumber
\mid \mathcal{T}_3(1)\mid \leq \left(1+\frac{B_{1}}{q+q^{2}}\times\left(\frac{B_{1}}{q}\left(1+2q\right)-\frac{B_{2}}{B_{1}}\right)\right)\nonumber
\times\left(1+\frac{B_{1}}{q+q^{2}}\max\left\{1, \left|\frac{B_{2}}{B_{1}}+\frac{B_{1}}{q}\right|\right\}\right).
\end{align}
When $q\rightarrow1^{-}$,
\begin{align}\nonumber
\mid \mathcal{T}_3(1)\mid &\leq \left(1+\frac{B_{1}}{2}\times\left(3B_{1}-\frac{B_{2}}{B_{1}}\right)\right)
\times\left(1+\frac{B_{1}}{2}\max\left\{1, \left|\frac{B_{2}}{B_{1}}+B_{1}\right|\right\}\right).
\end{align}
\end{corollary}

\begin{corollary}
If $f$ is given by (\ref{e1.1}), and if $f\in\mathcal{S}^{\ast }_{\alpha, \beta, \gamma}(\phi)$, then
\begin{align}\nonumber
\mid \mathcal{T}_3(1)\mid &\leq \left(1+\frac{B_{1}}{[3]_{\alpha, \beta, \gamma}-1}\times\left(\frac{B_{1}}{[2]_{\alpha, \beta, \gamma}-1}\left(2\frac{[3]_{\alpha, \beta, \gamma}-1}{[2]_{\alpha, \beta, \gamma}-1}-1\right)-\frac{B_{2}}{B_{1}}\right)\right)\\\nonumber
&\times\left(1+\frac{B_{1}}{[3]_{\alpha, \beta, \gamma}-1}\max\left\{1, \left|\frac{B_{2}}{B_{1}}+\frac{B_{1}}{[2]_{\alpha, \beta, \gamma}-1}\right|\right\}\right).
\end{align}
\end{corollary}
Next, we estimate the second-kind Hankel determinant $\mathcal{H}_2(1)$ for the function class $\mathcal{S}^{\ast }_{g}(\phi)$, leading to the following theorem and corollaries. As a direct consequence of Theorem 2.10, we derive the following theorem.
\begin{theorem}
Assume that $f$ is given by (\ref{e1.1}). If $f$ belongs to the function class $\mathcal{S}^{\ast }_{g}(\phi)$, then
\begin{align}\nonumber
\mid \mathcal{H}_2(1)\mid &\leq \frac{B_{1}}{[3]_{g}-1}\max\left\{1,\left|\frac{B_{2}}{B_{1}}+
\frac{B_{1}}{[2]_{g}-1}\left(1-\frac{[3]_{g}-1}{[2]_{g}-1}\right)\right|\right\}.
\end{align}

\end{theorem}
Similarly, by applying different $g(z)$ from Remark 1.2 to Theorem 2.25, we acquire the following corollaries.
\begin{corollary}
If $f$ is defined by (\ref{e1.1}), and if $f\in\mathcal{S}^{\ast }_{q}(\phi)$ , then
\begin{align}\nonumber
\mid \mathcal{H}_2(1)\mid &\leq \frac{B_{1}}{q(1+q)}\max\left\{1,\left|\frac{B_{2}}{B_{1}}-B_{1}\right|\right\}.
\end{align}
When $q\rightarrow1^{-}$,
\begin{align}\nonumber
\mid \mathcal{H}_2(1)\mid &\leq \frac{B_{1}}{2}\max\left\{1,\left|\frac{B_{2}}{B_{1}}-B_{1}\right|\right\}.
\end{align}
\end{corollary}

\begin{corollary}
If $f$ is defined by (\ref{e1.1}), and if $f\in\mathcal{S}^{\ast }_{\alpha, \beta, \gamma}(\phi)$ , then
\begin{align}\nonumber
\mid \mathcal{H}_2(1)\mid &\leq \frac{B_{1}}{[3]_{\alpha, \beta, \gamma}-1}\max\left\{1,\left|\frac{B_{2}}{B_{1}}+
\frac{B_{1}}{[2]_{\alpha, \beta, \gamma}-1}\left(1-\frac{[3]_{\alpha, \beta, \gamma}-1}{[2]_{\alpha, \beta, \gamma}-1}\right)\right|\right\}.
\end{align}
\end{corollary}

\section{The Class $\mathcal{C}_{g}(\phi)$ Of Convex Functions }
\setcounter{equation}{0} \setcounter{theorem}{0}

Below, we conduct coefficient estimations for the class $\mathcal{C}_{g}(\phi)$ of convex functions and further investigate its Fekete-Szeg\"{o} inequalities, Toeplitz determinants $\mathcal{T}_2(2)$, and $\mathcal{T}_3(1)$. Next, we will study the functions class $\mathcal{C}_{g}(\phi)$ and derive some theorems and corollaries.
\begin{definition}
If $f\in \mathbb{A} $ satisfies the following subordination relation
\begin{align}
\frac{D_{g}(zD_{g}f(z))}{D_{g}f(z)}\prec \phi(z),\ \ z\in\mathbb{U}\nonumber,
\end{align}
where $g\in \mathbb{A}$, then $f$ is said to belong to $f\in\mathcal{C}_{g}(\phi)$.

For simplicity, the function $f$ is said to belong to $f\in\mathcal{C}_{\alpha, \beta, \gamma}(\phi)$ when $g(z)=\frac{2(1-\gamma)z}{(1-\alpha z)(1-\beta z)}$ in Remark 1.2, then $f\in\mathcal{C}_{p,q}(\phi)$, $f\in\mathcal{C}_{q}(\phi)$ and $f\in\mathcal{C}_{1^{-}}(\phi)$ respectively when $g(z)$ corresponds to the cases (1), (2), and (3) in this Remark.

\end{definition}

Example related to the function class $\mathcal{C}_{g}(\phi)$ is subsequently provided.
\begin{example}
A function $f\in\mathcal{C}_{1^{-}}(\phi)$ if and only if there exists an analytic function $h\prec\phi(z)$ such that
\begin{align}\nonumber
f(z)=\int^{z}_{0} \exp \left(\int^{\omega}_{0}\frac{h(t)-1}{t}dt\right)d\omega,\ \ \ z\in\mathbb{U}.
\end{align}
This integral representation provides many examples of functions in the class $\mathcal{C}_{1^{-}}(\phi)$. Let
\begin{align}
h(t)=\phi(t^{n})=1+t^{n},\nonumber
\end{align}
where $n\in\mathbb{N}$, and $t\in\mathbb{U}$.
\begin{figure}[!ht]
  \centering
  \includegraphics[width=6cm]{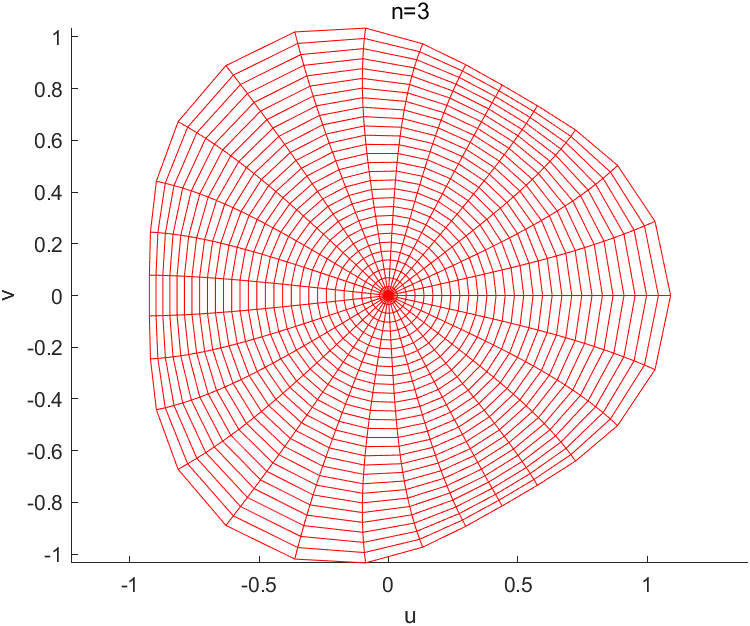}
  \includegraphics[width=6cm]{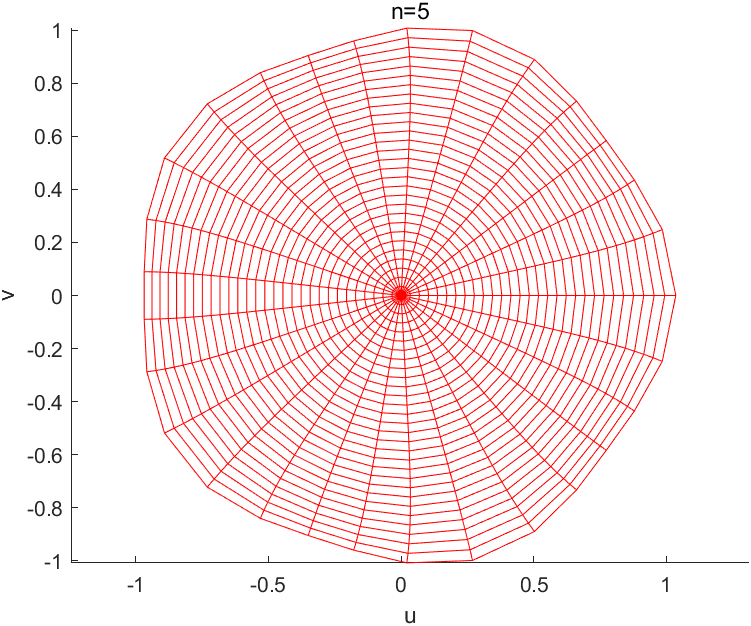}
  \caption{The image of $\mathbb{U}$ under $\psi(z)$ for $n=3, 5$.}
  \label{4}
\end{figure}

Through the straightforward computation, the function (as shown in Figure 2)
\begin{align}
\psi(z)=&\int^{z}_{0} \exp \left(\int^{\omega}_{0}\frac{q(t)-1}{t}dt\right)d\omega\nonumber
\\=&z+\frac{z^{n+1}}{n(n+1)}+\frac{z^{2n+1}}{2n^{2}(2n+1)}+\cdot\cdot\cdot\nonumber
\end{align}
is an extremal function in $\mathcal{C}_{1^{-}}(\phi)$.

\end{example}
Below, we study the coefficients and Fekete-Szeg\"{o} inequalities of the function class $\mathcal{C}_{g}(\phi)$. First, we estimate the initial coefficients of the function class $\mathcal{C}_{g}(\phi)$.
\begin{theorem}
Assume that $f$ is defined as in (\ref{e1.1}). If $f$belongs to the function class $\mathcal{C}_{g}(\phi)$, then
\begin{align}
\mid a_{2}\mid \leq \frac{B_{1}}{[2]_{g}([2]_{g}-1)}\nonumber
\end{align}
\begin{align}\nonumber
\mid a_{3}\mid \leq\frac{B_{1}}{[3]_{g}([3]_{g}-1)}\max\left\{1, \left|\frac{B_{2}}{B_{1}}+\frac{B_{1}}{[2]_{g}-1}\right|\right\}.\nonumber
\end{align}
\end{theorem}
\begin{proof}
Assume that $f\in\mathcal{C}_{g}(\phi)$. Then, there exists a Schwarz function $u(z)$ satisfying $u(0) = 0$ and $|u(z)| \leq 1$, which is analytic in $\mathbb{U}$, such that
\begin{align}\label{e2.10}
\frac{D_{g}(zD_{g}f(z))}{D_{g}f(z)}=\phi(u(z)).
\end{align}
Since $f\in \mathbb{A} $, the straightforward computation gives	
\begin{align}\label{e2.11}\nonumber
\frac{D_{g}(zD_{g}f(z))}{D_{g}f(z)}=&1+[2]_{g}\left([2]_{g}-1\right)a_{2}z\\
&+\left([3]_{g}([3]_{g}-1)a_{3}-[2]_{g}^{2}([2]_{g}-1)a_{2}^{2}\right)z^{2}+\cdot\cdot\cdot.
\end{align}
By combining the equations (\ref{e2.3}), (\ref{e2.10}), and (\ref{e2.11}), we obtain
\begin{align}
[2]_{g}\left([2]_{g}-1\right)a_{2} &= \frac{c_{1}B_{1}}{2}, \\
[3]_{g}([3]_{g}-1)a_{3}-[2]_{g}^{2}([2]_{g}-1)a_{2}^{2} &= \left(\frac{c_{2}}{2}-\frac{c_{1}^{2}}{4}\right)B_{1}+\frac{c_{1}^{2}}{4}B_{2}.
\end{align}
From (3.3) and (3.4), it follows that
\begin{align}
a_{2} &= \frac{c_1B_{1}}{2[2]_{g}([2]_{g}-1)}, \\
a_{3} &= \frac{B_{1}}{2[3]_{g}([3]_{g}-1)}\left[c_{2}-\frac{1}{2}\left(1-\frac{B_{2}}{B_{1}}-\frac{B_{1}}{[2]_{g}-1}\right)c_{1}^{2}\right].
\end{align}
Applying Lemmas 2.1 and 2.2 to (3.5) and (3.6), respectively, we derive
\begin{align}
\mid a_{2}\mid \leq \frac{B_{1}}{[2]_{g}([2]_{g}-1)}\nonumber
\end{align}
and
\begin{align}\nonumber
\mid a_{3}\mid \leq\frac{B_{1}}{[3]_{g}([3]_{g}-1)}\max\left\{1, \left|\frac{B_{2}}{B_{1}}+\frac{B_{1}}{[2]_{g}-1}\right|\right\},\nonumber
\end{align}
which completes the proof of Theorem 3.2.
\end{proof}

Next, substituting $g(z)$ from (2) Remark 1.2 into Theorem 3.2, we attain the following corollaries.
\begin{corollary}
Let $f$ be given by (1.1). If $f\in\mathcal{C}_{q}(\phi)$ , then
\begin{align}
\mid a_{2}\mid \leq \frac{B_{1}}{q(1+q)}\nonumber,
\ \
\mid a_{3}\mid \leq\frac{B_{1}}{q(1+q)(1+q+q^{2})}\max\left\{1, \left|\frac{B_{2}}{B_{1}}+\frac{B_{1}}{q}\right|\right\}.\nonumber
\end{align}
As $q\rightarrow1^{-}$,
\begin{align}\nonumber
\mid a_{2}\mid \leq \frac{B_{1}}{2},\quad\mid a_{3}\mid \leq\frac{B_{1}}{6}\max\left\{1, \left|\frac{B_{2}}{B_{1}}+B_{1}\right|\right\}.
\end{align}
\end{corollary}
Now we proceed to estimate Fekete-Szeg\"{o} functional inequality for the function class $\mathcal{C}_{g}(\phi)$.
\begin{theorem}
Assume that $f$ is given by (\ref{e1.1}). If $f\in\mathcal{C}_{g}(\phi)$, then for any complex number $\mu$,
\begin{align}\nonumber
\mid a_3-\mu a_2^2\mid\leq \frac{B_{1}}{[3]_{g}([3]_{g}-1)}\max\left\{1,\left|\frac{B_{2}}{B_{1}}+
\frac{B_{1}}{[2]_{g}-1}\left(1-\frac{[3]_{g}([3]_{g}-1)}{[2]_{g}^{2}([2]_{g}-1)}\mu\right)\right|\right\}.
\end{align}
\end{theorem}
\begin{proof}
From (3.5) and (3.6), we have
\begin{align}\nonumber
a_{2}^{2} &= \frac{c_{1}^{2}B_{1}^{2}}{4[2]_{g}^{2}([2]_{g}-1)^{2}}, \\\nonumber
a_{3} &= \frac{B_{1}}{2[3]_{g}([3]_{g}-1)}\left[c_{2}-\frac{1}{2}\left(1-\frac{B_{2}}{B_{1}}-\frac{B_{1}}{[2]_{g}-1}\right)c_{1}^{2}\right].
\end{align}
Hence
\begin{align}\nonumber\label{e2.16}
a_{3}-\mu a_{2}^{2}&=\frac{B_{1}}{2[3]_{g}([3]_{g}-1)}\left[c_{2}-\frac{1}{2}\left(1-\frac{B_{2}}{B_{1}}-\frac{B_{1}}{[2]_{g}-1}\right)c_{1}^{2}\right]-\frac{\mu c_{1}^{2}B_{1}^{2}}{4[2]_{g}^{2}([2]_{g}-1)^{2}}\\
&=\frac{B_{1}}{2[3]_{g}([3]_{g}-1)}(c_{2}-\iota c_{1}^{2}),
\end{align}
where
\begin{align}\nonumber
\iota=\frac{1}{2}\left(1-\frac{B_{2}}{B_{1}}-
\frac{B_{1}}{[2]_{g}-1}\left(1-\frac{[3]_{g}([3]_{g}-1)}{[2]_{g}^{2}([2]_{g}-1)}\mu\right)\right).
\end{align}
Applying Lemma 2.2 to (\ref{e2.16}), we obtain
\begin{align}\nonumber
\mid a_3-\mu a_2^2\mid\leq \frac{B_{1}}{[3]_{g}([3]_{g}-1)}\max\left\{1,\left|\frac{B_{2}}{B_{1}}+
\frac{B_{1}}{[2]_{g}-1}\left(1-\frac{[3]_{g}([3]_{g}-1)}{[2]_{g}^{2}([2]_{g}-1)}\mu\right)\right|\right\}.
\end{align}
This finishes the proof of Theorem 3.4.
\end{proof}

Next, substituting $g(z)$ from (2) in Remark 1.2 into Theorem 3.4 yields the following corollaries.
\begin{corollary}
Let $f\in \mathbb{A} $ be given by (1.1). If $f\in\mathcal{C}_{q}(\phi)$ and $\mu \in \mathbb{C}$, then
\begin{align}\nonumber
\mid a_3-\mu a_2^2\mid\leq \frac{B_{1}}{q(1+q)(1+q+q^{2})}\max\left\{1,\left|\frac{B_{2}}{B_{1}}+
\frac{B_{1}}{q}\left(1-\frac{1+q+q^{2}}{1+q}\mu\right)\right|\right\}.
\end{align}
As $q\rightarrow1^{-}$,
\begin{align}\nonumber
\mid a_3-\mu a_2^2\mid\leq \frac{B_{1}}{6}\max\left\{1,\left|\frac{B_{2}}{B_{1}}+
B_{1}\left(1-\frac{3}{2}\mu\right)\right|\right\}.
\end{align}
\end{corollary}
If we take $\mu \in \mathbb{R}$, then based on the proof of Theorem 2.16 and Lemma 2.3, the Fekete-Szeg\"{o} inequality for the function class $\mathcal{C}_{g}(\phi)$ can be estimated.
\begin{theorem}
Let $f \in \mathcal{A}$ be given by (1.1), and $\mu \in \mathbb{R}$. If $f \in \mathcal{C}_{g}(\phi)$, then
\begin{equation*}
\left|a_{3} - \mu a_{2}^{2}\right| \leq \frac{B_{1}}{[3]_{g}([3]_{g}-1)}
\begin{cases}
-2\iota + 1, & \mu \leq \iota_{1}, \\
1, &  \iota_{1} \leq \mu \leq \iota_{2}, \\
2\iota - 1, &  \mu \geq \iota_{2},
\end{cases}
\end{equation*}
where
\begin{align}\nonumber
\iota_{1}=\frac{B_{1}^{2}[2]_{g}^{2}([2]_{g}-1)+(B_{2}-B_{1})[2]_{g}^{2}([2]_{g}-1)^{2}}{B_{1}^{2}[3]_{g}([3]_{g}-1)},
\end{align}
\begin{align}\nonumber
\iota_{2}=\frac{B_{1}^{2}[2]_{g}^{2}([2]_{g}-1)+(B_{2}+B_{1})[2]_{g}^{2}([2]_{g}-1)^{2}}{B_{1}^{2}[3]_{g}([3]_{g}-1)}.
\end{align}
\end{theorem}
By applying $g(z)$ from (2) in Remark 1.2 into Theorem 3.6 yields the following corollaries.
\begin{corollary}
Let $f\in \mathbb{A} $ be given by (1.1). If $f\in\mathcal{C}_{q}(\phi)$ and $\mu \in \mathbb{R}$ , then
\begin{equation*}
\left|a_{3} - \mu a_{2}^{2}\right| \leq \frac{B_{1}}{q(1+q)(1+q+q^{2})}
\begin{cases}
-2\iota + 1, & \mu \leq \iota_{1}, \\
1, &  \iota_{1} \leq \mu \leq \iota_{2}, \\
2\iota - 1, &  \mu \geq \iota_{2},
\end{cases}
\end{equation*}
where
\begin{align}\nonumber
\iota_{1}=\frac{(1+q)(B_{1}^{2}+q(B_{2}-B_{1}))}{B_{1}^{2}(1+q+q^{2})},
\end{align}
\begin{align}\nonumber
\iota_{2}=\frac{(1+q)(B_{1}^{2}+q(B_{2}+B_{1}))}{B_{1}^{2}(1+q+q^{2})}.
\end{align}
When $q\rightarrow1^{-}$,
\begin{equation*}
\left|a_{3} - \mu a_{2}^{2}\right| \leq \frac{B_{1}}{6}
\begin{cases}
-2\iota + 1, & \mu \leq \iota_{1}, \\
1, &  \iota_{1} \leq \mu \leq \iota_{2}, \\
2\iota - 1, &  \mu \geq \iota_{2},
\end{cases}
\end{equation*}
where
\begin{align}\nonumber
\iota_{1}=\frac{2B_{1}^{2}+2(B_{2}-B_{1})}{3B_{1}^{2}},
\end{align}
\begin{align}\nonumber
\iota_{2}=\frac{2B_{1}^{2}+2(B_{2}+B_{1})}{3B_{1}^{2}}.
\end{align}
\end{corollary}
Subsequently, the Toeplitz determinant of the function class $\mathcal{C}_{g}(\phi)$ is studied, starting with an estimation of $\mathcal{T}_2(2)$.
\begin{theorem}
Suppose that $f$ is given by (\ref{e1.1}). If $f\in\mathcal{C}_{g}(\phi)$, then

(1) If $|\frac{B_{2}}{B_{1}}+\frac{B_{1}}{[2]_{g}-1}|\leq1$, then
\begin{align}\nonumber
\mid \mathcal{T}_2(2)\mid=\mid a_{2}^2-a_{3}^{2}\mid=
\frac{B_{1}^{2}}{[2]_{g}^{2}([2]_{g}-1)^{2}}-\frac{B_{1}^{2}}{[3]_{g}^{2}([3]_{g}-1)^{2}}.
\end{align}

(2) If $|\frac{B_{2}}{B_{1}}+\frac{B_{1}}{[2]_{g}-1}|>1$, then
\begin{align}\nonumber
\mid \mathcal{T}_2(2)\mid=\mid a_{2}^2-a_{3}^{2}\mid\leq
\frac{B_{1}^{2}}{[2]_{g}^{2}([2]_{g}-1)^{2}}+\frac{B_{1}^{2}}{[3]_{g}^{2}([3]_{g}-1)^{2}}
\times\left|\frac{B_{2}}{B_{1}}+\frac{B_{1}}{[2]_{g}-1}\right|^{2}.
\end{align}
\end{theorem}
By substituting $g(z)$ from Remark 1.2 into Theorem 3.8, we obtain the following corollaries.
\begin{corollary}
If $f$ is given by (\ref{e1.1}) and $f\in\mathcal{C}_{q}(\phi)$ , then

(1) If $|\frac{B_{2}}{B_{1}}+\frac{B_{1}}{q}|\leq1$, then
\begin{align}\nonumber
\mid \mathcal{T}_2(2)\mid=\mid a_{2}^2-a_{3}^{2}\mid=
\frac{B_{1}^{2}}{q(1+q)^{2}}\left(\frac{2+3q+2q^{2}+q^{3}}{(1+q+q^{2})^{2}}\right).
\end{align}

(2) If $|\frac{B_{2}}{B_{1}}+\frac{B_{1}}{q}|>1$, then
\begin{align}\nonumber
\mid \mathcal{T}_2(2)\mid=\mid a_{2}^2-a_{3}^{2}\mid\leq
\frac{B_{1}^{2}}{q^{2}(1+q)^{2}}\left(1+\frac{1}{(1+q+q^{2})^{2}}
\times\left|\frac{B_{2}}{B_{1}}+\frac{B_{1}}{q}\right|^{2}\right).
\end{align}
Let $q\rightarrow1^{-}$,

(1) If $|\frac{B_{2}}{B_{1}}+B_{1}|\leq1$, then
\begin{align}\nonumber
\mid \mathcal{T}_2(2)\mid=\mid a_{2}^2-a_{3}^{2}\mid\leq
\frac{2B_{1}^{2}}{9}.
\end{align}

(2) If $|\frac{B_{2}}{B_{1}}+B_{1}|>1$, then
\begin{align}\nonumber
\mid \mathcal{T}_2(2)\mid=\mid a_{2}^2-a_{3}^{2}\mid\leq
\frac{B_{1}^{2}}{4}+\frac{B_{1}^{2}}{36}
\times\left|\frac{B_{2}}{B_{1}}+B_{1}\right|^{2}.
\end{align}
\end{corollary}
Based on Theorem 3.4 and the proof of Lemma 2.2, we consider the symmetric Toeplitz determinant $\mathcal{T}_3(1)$ of $\mathcal{C}_{g}(\phi)$ and derive the following theorem.
\begin{theorem}
Let $f$ be given by (\ref{e1.1}). If $f\in\mathcal{C}_{g}(\phi)$, then
\begin{align}\nonumber
\mid \mathcal{T}_3(1)\mid &\leq \left(1+\frac{B_{1}}{[3]_{g}([3]_{g}-1)}\max\left\{1,\left|\frac{B_{2}}{B_{1}}+
\frac{B_{1}}{[2]_{g}-1}\left(1-\frac{2[3]_{g}([3]_{g}-1)}{[2]_{g}^{2}([2]_{g}-1)}\right)\right|\right\}\right)\\\nonumber
&\times\left(1+\frac{B_{1}}{[3]_{g}([3]_{g}-1)}\max\left\{1, \left|\frac{B_{2}}{B_{1}}+\frac{B_{1}}{[2]_{g}-1}\right|\right\}\right).
\end{align}
\end{theorem}

\begin{proof}

Applying Theorem 3.4 and Lemma 2.2 to the inequality (2.10), we obtain
\begin{align}\nonumber
\mid \mathcal{T}_3(1)\mid &\leq \left(1+\frac{B_{1}}{[3]_{g}([3]_{g}-1)}\max\left\{1,\left|\frac{B_{2}}{B_{1}}+
\frac{B_{1}}{[2]_{g}-1}\left(1-\frac{2[3]_{g}([3]_{g}-1)}{[2]_{g}^{2}([2]_{g}-1)}\right)\right|\right\}\right)\\\nonumber
&\times\left(1+\frac{B_{1}}{[3]_{g}([3]_{g}-1)}\max\left\{1, \left|\frac{B_{2}}{B_{1}}+\frac{B_{1}}{[2]_{g}-1}\right|\right\}\right).
\end{align}
Thus, Theorem 3.10 holds.
\end{proof}
By taking  $g(z)$ from (2) in Remark 1.2 into Theorem 3.10, we acquire the following corollaries.
\begin{corollary}
If $f$ is given by (\ref{e1.1}) and $f\in\mathcal{C}_{q}(\phi)$ , then
\begin{align}\nonumber
\mid \mathcal{T}_3(1)\mid &\leq \left(1+\frac{B_{1}}{q(1+q)(1+q+q^{2})}\max\left\{1,\left|\frac{B_{2}}{B_{1}}-
\frac{B_{1}(1+q+2q^{2})}{q(1+q)}\right|\right\}\right)\\\nonumber
&\times\left(1+\frac{B_{1}}{q(1+q)(1+q+q^{2})}\max\left\{1, \left|\frac{B_{2}}{B_{1}}+\frac{B_{1}}{q}\right|\right\}\right).
\end{align}
When $q\rightarrow1^{-}$,
\begin{align}\nonumber
\mid \mathcal{T}_3(1)\mid \leq \left(1+\frac{B_{1}}{6}\max\left\{1,\left|\frac{B_{2}}{B_{1}}-
2B_{1}\right|\right\}\right)\times\left(1+\frac{B_{1}}{6}\max\left\{1, \left|\frac{B_{2}}{B_{1}}+B_{1}\right|\right\}\right).
\end{align}
\end{corollary}
Similarly, based on the proof of Theorem 3.6 and Lemma 2.3, we consider the symmetric Toeplitz determinant $\mathcal{T}_3(1)$ of $\mathcal{C}_{g}(\phi)$ and establish the following theorem.
\begin{theorem}
Let $f$ be given by (\ref{e1.1}). If $f\in\mathcal{C}_{g}(\phi)$, then
\begin{align}\nonumber
\mid \mathcal{T}_3(1)\mid &\leq \left(1+\frac{B_{1}}{[3]_{g}([3]_{g}-1)}\times\left(\frac{B_{1}}{[2]_{g}-1}\left(2\frac{[3]_{g}([3]_{g}-1)}{[2]_{g}^{2}([2]_{g}-1)}-1\right)-\frac{B_{2}}{B_{1}}\right)\right)\\\nonumber
&\times\left(1+\frac{B_{1}}{[3]_{g}([3]_{g}-1)}\max\left\{1, \left|\frac{B_{2}}{B_{1}}+\frac{B_{1}}{[2]_{g}-1}\right|\right\}\right).
\end{align}
\end{theorem}
\begin{proof}
By the determinant $\mathcal{T}_3(1)$ and Theorem 2.18,
\begin{align}\nonumber
\mid a_{3}-2a_{2}^{2}\mid & \leq \frac{B_{1}}{2([3]_{g}-1)}(4\iota - 2) \\\nonumber
&\leq\frac{B_{1}}{[3]_{g}([3]_{g}-1)}\times\left(\frac{B_{1}}{[2]_{g}-1}\left(2\frac{[3]_{g}([3]_{g}-1)}{[2]_{g}^{2}([2]_{g}-1)}-1\right)-\frac{B_{2}}{B_{1}}\right).
\end{align}
Hence
\begin{align}\nonumber
\mid \mathcal{T}_3(1)\mid &\leq \left(1+\frac{B_{1}}{[3]_{g}([3]_{g}-1)}\times\left(\frac{B_{1}}{[2]_{g}-1}\left(2\frac{[3]_{g}([3]_{g}-1)}{[2]_{g}^{2}([2]_{g}-1)}-1\right)-\frac{B_{2}}{B_{1}}\right)\right)\\\nonumber
&\times\left(1+\frac{B_{1}}{[3]_{g}([3]_{g}-1)}\max\left\{1, \left|\frac{B_{2}}{B_{1}}+\frac{B_{1}}{[2]_{g}-1}\right|\right\}\right).
\end{align}
Thus, Theorem 3.12 holds.
\end{proof}
Substituting $g(z)$ from (2) in Remark 1.2 into Theorem 3.12, we get the following corollaries.
\begin{corollary}
If $f$ is given by (\ref{e1.1}) and $f\in\mathcal{C}_{q}(\phi)$, then
\begin{align}\nonumber
\mid \mathcal{T}_3(1)\mid &\leq \left(1+\frac{B_{1}}{q(1+q)(1+q+q^{2})}\times\left(\frac{B_{1}(1+q+2q^{2})}{q(1+q)}-\frac{B_{2}}{B_{1}}\right)\right)\\\nonumber
&\times\left(1+\frac{B_{1}}{q(1+q)(1+q+q^{2})}\max\left\{1, \left|\frac{B_{2}}{B_{1}}+\frac{B_{1}}{q}\right|\right\}\right).
\end{align}
When $q\rightarrow1^{-}$,
\begin{align}\nonumber
\mid \mathcal{T}_3(1)\mid &\leq \left(1+\frac{B_{1}}{6}\times\left(2B_{1}-\frac{B_{2}}{B_{1}}\right)\right)
\times\left(1+\frac{B_{1}}{6}\max\left\{1, \left|\frac{B_{2}}{B_{1}}+B_{1}\right|\right\}\right).
\end{align}
\end{corollary}
Finally, we estimate the second Hankel determinant $\mathcal{H}_2(1)$ for the function class $\mathcal{C}_{g}(\phi)$, leading to the following theorem and corollary. As a consequence of Theorem 3.4, immediately we give the theorem below.
\begin{theorem}
Let $f$ be given by (\ref{e1.1}). If $f\in\mathcal{C}_{g}(\phi)$, then
\begin{align}\nonumber
\mid \mathcal{H}_2(1)\mid &\leq \frac{B_{1}}{[3]_{g}([3]_{g}-1)}\max\left\{1,\left|\frac{B_{2}}{B_{1}}+
\frac{B_{1}}{[2]_{g}-1}\left(1-\frac{[3]_{g}([3]_{g}-1)}{[2]_{g}^{2}([2]_{g}-1)}\right)\right|\right\}.
\end{align}
\end{theorem}
\begin{corollary}
If $f$ is given by (\ref{e1.1}) and $f\in\mathcal{C}_{q}(\phi)$ , then
\begin{align}\nonumber
\mid \mathcal{H}_2(1)\mid &\leq \frac{B_{1}}{q(1+q)(1+q+q^{2})}\max\left\{1,\left|\frac{B_{2}}{B_{1}}-
\frac{qB_{1}}{1+q}\right|\right\}.
\end{align}
When $q\rightarrow1^{-}$,
\begin{align}\nonumber
\mid \mathcal{H}_2(1)\mid &\leq \frac{B_{1}}{6}\max\left\{1,\left|\frac{B_{2}}{B_{1}}-\frac{B_{1}}{2}\right|\right\}.
\end{align}
\end{corollary}

\section{ Conclusions }
\setcounter{equation}{0} \setcounter{theorem}{0}

In this paper the most innovation is that we unify all derivative operators, even classical derivative, in the literature. Similarly, in the sense of the $g$-derivative operator
the primary objective of this research is to determine the upper bounds of the coefficients $a_{2}$ and $a_{3}$ for two function classes $\mathcal{S}^{\ast }_{g}(\phi)$ and $\mathcal{C}_{g}(\phi)$. 
Additionally, the related results for Fekete-Szeg\"{o} functional inequalities, Toeplitz determinants and Hankel determinants are proved. Correspondingly, some corollaries are derived for various $g(z)$. Besides, drawing upon the findings presented in this paper, 
the investigations of Bohr radius can be pursued. In the future, researchers may further explore functional inequalities for another classes of special functions by utilizing the $g$-derivative operator as defined herein.


\bigskip

\noindent An Huang, Pinhong Long

\medskip

\noindent School of Mathematics and Statistics, Ningxia
University, Yinchuan, Ningxia 750021, People's Republic of China

\smallskip

\noindent{\it E-mail:} \texttt{longph@nxu.edu.cn}

\bigskip

\noindent Halit Orhan

\medskip

\noindent Department of Mathematics, Faculty of Science,  Ataturk University, Erzurum 25240, Turkey

\smallskip

\noindent{\it E-mail:} \texttt{orhanhalit607@gmail.com}

\bigskip

\noindent Huo Tang

\medskip

\noindent School of Mathematics and Computer Sciences, Chifeng University, Chifeng, Inner Mongolia, 024000, China

\smallskip

\noindent{\it E-mail:} \texttt{thth2009@163.com}

\end{document}